\documentclass[11pt,english,twoside,a4paper]{amsart}     
\usepackage[english]{babel}
\usepackage{amssymb,amsmath,amsthm,stmaryrd}
\usepackage{mathrsfs}
\usepackage{graphicx,psfrag,epsfig}
\usepackage{graphicx}
\usepackage[T1]{fontenc}
\usepackage[latin1]{inputenc}
\usepackage[all]{xy}
\usepackage{mathrsfs}
\usepackage{pst-all}
\usepackage{float}
\theoremstyle{plain}

\theoremstyle{plain}

\newtheorem{thm}{Theorem}
\newtheorem{prop}{Proposition}[section]
\newtheorem{lem}{Lemma}[section]

\newtheorem{nota}{Notation}[section]

\newtheorem{cor}{Corollary}[section]

\theoremstyle{definition}
\newtheorem{Def}{Definition}[section]
\newtheorem{exm}{Example}[section]
\theoremstyle{remark}
\newtheorem{rem}{Remark}[section]
\newtheoremstyle{restate}
{\topsep}
{\topsep}
{\itshape}
{}
{\bfseries}
{.}
{ }
{\thmname{#1}\thmnote{ #3}}
\theoremstyle{restate}

\newtheorem{thm*}{Theorem}
\newtheorem{prop*}{Proposition}
\newtheorem{lem*}{Lemma}
\newtheorem{add*}{Addendum}
\newtheorem{cor*}{Corollary}
\newtheorem{pte*}{Property}
\theoremstyle{definition}
\newtheorem{Def*}{Definition}
\newtheorem{exm*}{Example}
\theoremstyle{remark}
\newtheorem{rem*}{Remark}

\theoremstyle{definition}

\theoremstyle{remark}

\newcommand{\Id}{\mathop{\hbox{{\rm Id}}}\nolimits}

\newcommand{\Inf}{\mathop{\rm Inf\,}\limits}

\newcommand{\Dp}[2]{\frac{\partial #1}{\partial #2}}

\newcommand{\ha}{\widehat}

\newcommand{\wti}{\widetilde}
\newcommand{\N}{\mathbb N}
\newcommand{\Z}{\mathbb Z}

\newcommand{\R}{\mathbb R}

\newcommand{\T}{\mathbb T}

\newcommand{\U}{\mathbb U}
\newcommand{\Cal}{\mathcal}

\newcommand{\mcr}{\mathscr}

\newcommand{\demi}{\frac{1}{2}}

\newcommand{\Log}{{\rm Log\,}}
\newcommand{\setm}{\setminus}
\DeclareMathOperator{\Vol}{Vol}

\newcommand{\Aa}{\Cal{A}}
\newcommand{\Ss}{\Cal S}
\newcommand{\Tt}{\Cal T}
\newcommand{\Dd}{\Cal D}

\newcommand{\Cc}{\Cal C}

\newcommand{\Pp}{\Cal P}
\newcommand{\Hh}{\Cal H}

\newcommand{\Rr}{\Cal R}

\newcommand{\Ii}{\Cal I}

\newcommand{\jA}{\mcr A}
\newcommand{\CC}{\mcr C}
\newcommand{\TT}{\mcr T}
\newcommand{\EE}{\mcr E}

\newcommand{\DD}{\mcr D}
\newcommand{\LL}{\mcr L}
\newcommand{\PP}{\mcr P}

\newcommand{\II}{\mcr I}

\newcommand{\MM}{\mcr M}

\newcommand{\UU}{\mcr U}

\newcommand{\RR}{\mcr R}

\newcommand{\al}{\alpha}

\newcommand{\ga}{\gamma}
\newcommand{\sig}{\sigma}
\newcommand{\eps}{\varepsilon}
\newcommand{\Ga}{\Gamma}
\renewcommand{\th}{\theta}

\newcommand{\Sig}{\Sigma}

\newcommand{\om}{\omega}
\newcommand{\Om}{\Omega}
\newcommand{\lam}{\lambda}

\newcommand{\de}{\delta}

\newcommand{\rit}{\rightarrow}
\newcommand{\ma}{\mapsto}

\newcommand{\inv}{^{-1}}

\newcommand{\hto}{{\rm{h_{top}}}}
\newcommand{\hp}{{\rm{h_{pol}}}}

\newcommand{\hvo}{{\rm{h_{vol}}}}
 \newcommand{\8}{\scalebox{1.3}{$\infty$}}

\newcommand{\pp}{^{++}}
\newcommand{\pl}{^{+-}}
\newcommand{\lp}{^{-+}}
\newcommand{\mm}{^{--}}
\newcommand{\sss}{^{**}}

\begin{document}
\selectlanguage{english}

\author{Clémence Labrousse}
\title[Flat metrics minimize the polynomial entropy]{Flat metrics are strict local minimizers for the polynomial entropy}
\thanks{Institut de Mathématiques de Jussieu, UMR 7586, {\em Analyse algébrique},
175 rue du Chevaleret, 75013 Paris.
email: labrousse@math.jussieu.fr}

\date{}

\maketitle

\begin{abstract} 
As we have proved in \cite{L-1}, the geodesic flows associated with the flat metrics on $\T^2$  minimize the polynomial entropy $\hp$. In this paper, we show that, among the geodesic flows that are Bott integrable and dynamically coherent, the geodesic flows associated to flat metrics are local \emph{strict} minima for $\hp$. To this aim, we prove a graph property for invariant Lagrangian tori in near-integrable systems.
\end{abstract}


\section{Introduction}
Let $(M,g)$ be a compact Riemannian manifold. The \emph{Hamiltonian geodesic flow} $\phi_g$ is the flow associated with the geodesic Hamiltonian $H$ on $T^*M$ defined by
\[
\begin{matrix}
H\;\; : & T^*M & \longrightarrow &  \R\\
  & (m,p) & \mapsto & g_m^*(p,p),
 \end{matrix}
\]
where $g_m^*$ is deduced from $g_m$ by the Legendre transform. Throughout this paper, we will  consider only the restriction of $\phi_g$ to the unit cotangent bundle of $M$, which we still denote by $\phi_g$.

There are several ways to measure the \emph{complexity} of the geodesic flow $\phi_g$, one of them being the growth rate of the volume of balls in the Riemannian covering $\wti{M}$ of $M$. We recall that the volumic entropy $\hvo(g)$ of $(M,g)$ is defined as: 
\[
\hvo(g):=\limsup_{r\rit\infty}\frac{1}{r}\Log\Vol B(x,r),
\]
where $B(x,r)$ is the ball in $\wti{M}$ centered at $x$ and  of radius $r$. 
Another way to determine the complexity of $\phi_g$ is the topological entropy. Given a compact metric space $(X,d)$ and a continuous flow $\phi:\R\times X\rit X: (t,x)\ma\phi_t(x)$, for each $t>0$, one  defines the dynamical metric
\begin{equation}
d^{\phi}_t(x,y)=\max_{0\leq k\leq t}d(\phi_t(x),\phi_t(y)).
\end{equation}
All the metrics $d^{\phi}_t$ are equivalent to $d$. In particular, $(X,d^{\phi}_t)$ is compact. So, for any $\varepsilon>0$, $X$ can be covered by a finite number of balls of radius $\varepsilon$ for $d^{\phi}_t$. Let $G_t(\varepsilon)$ be the minimal cardinal of such a covering.
The topological entropy of $\phi$ is defined by:
\[
\hto(\phi)=\lim_{\varepsilon\rightarrow 0}\limsup_{t\rightarrow\infty}\frac{1}{t}\log G_t(\varepsilon).
\]
Manning proved in \cite{M-79} that $\hvo(g)\leq \hto(\phi_g)$, with equality if $(M,g)$ has sectional  non positive curvature.

Besson, Courtois and Gallot (\cite{BCG-95},\cite{BCG-96}) proved that locally symetric metrics of negative curvature (and in particular hyperbolic metrics) are strict minima for $\hto$. This result was already proved in dimension $2$ by Katok in \cite{Kat}.
Therefore, one knows the ``simplest'' metrics for surfaces with genus $\geq 2$ and we would like to find the metrics that minimize the complexity for surfaces with genus $1$. 

A first remark is that the volume growth is quadratic, and that the topological entropy may vanish. Indeed, a flat metric on $\T^2$ is an equality case in Manning's formula. Hence, since the growth of the volume is quadratic, the entropy is zero. Flat metrics are not the only ones for which the topological entropy vanishes. Indeed, Paternain proved in \cite{Pat-dim4} that if the geodesic flow admits a first integral $f$ such that the critical points of $f$ (in restriction to the unit cotangent bundle) form strict submanifolds, then its topological entropy vanishes.

This leads us to consider a ``polynomial measure'' of the complexity, namely the \emph{polynomial entropy} $\hp$, defined below, and to search the minimizers of $\hp$ among the metrics $g$ whose geodesic flow $\phi_g$ possesses a first integral that satisfies Paternain's hypotheses.

 The polynomial entropy is defined as follows:
\[
\hp(\phi)=\lim_{\varepsilon\rightarrow 0}\limsup_{t\rightarrow\infty}\frac{1}{\log t}\log G_t(\varepsilon)=\Inf\left\{\sig\geq 0\,|\, \lim_{t\to\infty}\frac{1}{t^\sig}G_t^\phi(\eps)=0\right\},
\]
where, as before, $\phi$ is a flow on a compact metric space $X$. We refer to \cite{Mar-09} for a complete introduction. 
In \cite{LM}, we proved that if $H$ is a Hamiltonian system on $T^*\T^n$ (with the canonical symplectic form $d\th\wedge dr$)  which is in action-angle form (that is, $H(\th,r)=h(r)$, for a smooth map $h:\R^n\rit \R$) then the polynomial entropy in restriction to a compact energy level $\EE=H\inv(\{e\})$ of $H$ is
\[
\hp(\phi_H,\EE)=\max_{r\in h\inv(\{e\})}\partial^2\left(h_{|h\inv(\{e\})}\right)(r).
\]
An immediate consequence of this result is that if $g_0$ is a flat metric on $\T^n$, then $\hp(\phi_{g_0})=n-1$.

Now, for a Riemannian manifold $(M,g)$ we set
\[
\tau(M)=\limsup_{r\rit\infty}\frac{\log \Vol B(x,r) }{\log r }=\inf\left\{s\geq 0\,|\,\limsup_{r\rit\infty}\frac{1}{r^s}\Vol B(x,r)=0\right\}.
\]
One checks that $\tau(M)$ is independent of $x$. Indeed, it is the  degree of growth of the fundamental group $\pi_1(M)$. In \cite{L-1}, we proved that, denoting by $\phi_g$ the geodesic flow of $g$:
\[
\tau(M)\leq \hp(\phi_g)+1.
\]
Therefore, since $\tau(\T^n)=n$ , the flat metrics on $\T^n$ do minimize $\hp$.

Finally, in \cite{LM}, we showed that if a Hamiltonian flow $\phi_H$ on a $4$-dimen\-sional symplectic manifold possesses a first integral, independent of $H$, that is \emph{nondegenerate in the Bott sense} on a compact regular level $\EE$ of $H$ and that satisfies an additionnal property of \,``dynamical coherence'', then $\hp(\phi_H,\EE)$ belongs to $\{0,1,2\}$ and $\hp(\phi_H,\EE)=2$ if and only if $\phi_H$ possesses a hyperbolic orbit. This result and the definition of dynamical coherence are recalled in Section 2. 
The main result of this paper is the following. We denote by $\DD\CC$ the set of metrics on $\T^2$ with dynamically coherent geodesic flows.
\vspace{0.4cm}

\noindent \textbf{Theorem A.}
\emph{Let} $g_0$ \emph{be a} \emph{flat metric on} $\T^2$. 
\emph{There exists a neighborhood} $\UU$ \emph{of} $g_0$ \emph{in the set of} $C^5$ 
\emph{metrics such that for any} $g\in \UU\cap\DD\CC$:
\begin{itemize}
\item \emph{either} $g$ \emph{is flat},
\item \emph{or} $g$ \emph{possesses a hyperbolic orbit.}
\end{itemize}
\vspace{0.2cm}

\noindent Therefore, if $g\in \UU\cap\DD\CC$ is not flat, then $\hp(\phi_g)=2>\hp(\phi_{g_0})$.

\section{Dynamically coherent systems}

Consider a  $4$-dimensional symplectic manifold $(M,\Om)$ and a smooth Hamiltonian function $H:M\rit \R$, with its associated vector field $X^H$ and its associated Hamiltonian flow $\phi_H$.
We fix a (connected component of ) a compact regular energy level $\EE$ of $H$. It is an orientable compact connected submanifold of dimension $3$.  

\begin{Def}
A first integral $F:M\rit \R$ of the vector field $X^H$ is said to be \emph{nondegenerate in the Bott sense} on $\EE$ if the critical points of $f:=F_{|\EE}$ form nondegenerate strict smooth submanifolds of $\EE$, that is, the Hessian $\partial^2f$ of $f$ is nondegenerate on complementary subspaces to these submanifolds. The triple $(\EE,\phi_H,f)$\index{$(\EE,\phi_H,f)$} is called a \emph{nondegenerate Bott system}.
\end{Def}

By the Arnol'd-Liouville theorem, if the regular levels of the moment map $(H,F)$ are compact,  the regular set of $(H,F)$ admits a covering by domains $\ha{\jA}$ saturated for $(H,F)$ and symplectomorphic to $\T^2\times B$ ($B\subset \R^2$), namely the \emph{action-angle domains}.  On such a domain,  the leaves 
of the foliation induced by $(H,F)$ are homotopic tori that are symplectomorphic to the tori $\T^2\times\{r\},\, r\in B$ and the flow is conjugated to a Hamiltonian flow $\phi$ on $\T^2\times B$  of the form $\phi^t(\th,r)=(\th+\om(r)\, [\Z^2],r)$. 

In the following, we consider  the restrictions of the vector field and  the flow to $\EE$, they are still denoted by $X^H$ and $\phi_H$. 
Let us describe the geometry of the level sets of $f$. We denote by $\Rr(f)$ the set of regular values of $f$ and by ${\rm{Crit}}(f)$ the set of its critical values. 
If $c\in{\rm{Crit}}(f)$, we denote by $\RR_c$ the union  of the connected components of $f\inv(\{c\})$ that does not contain any critical point. We define the \emph{regular set} of $f$ as $\RR:=f\inv(\{\Rr\})\cup\left(\bigcup_{c\in{\rm{Crit}}(f)}\RR_c\right)$. 
A connected component of $\RR$ is contained in the connected component of an intersection $\jA:=\ha{\jA}\cap\EE$, where $\ha\jA$ is an action-angle domain. 

The singularities of $f$ are well known. The following proposition is proved in \cite{Mar-93} and \cite{F-88}.

\begin{prop}\label{CritBott}
The critical submanifolds may only be circles, tori or Klein bottles.
\end{prop}


The  critical circles for $f$ are periodic orbits of the flow $\phi^H$. Their \textit{index} is the number of negative eigenvalues of the restriction of $\partial^2f$ to a supplementary plane to $\R X^H$.
Let us summarize briefly the two possibilities that occur (see \cite{Mar-93} for more details).
Fix a critical circle $\CC$ for $f$ wih $f(\CC)=c$. 

$\bullet$ If $\CC$ has index $0$ or $2$, there exists a neighborhood $U$ of $\CC$, saturated for $f$,  that is diffeomorphic to the full torus $D\times \T$ (where $D$ is an open disk in $\R^2$ containing $0$) such that $f\inv\{c\}\cap U=\CC$. Denoting by $\Psi$ the diffeomorphism, we assume that  $\CC=\Psi(\{0\}\times\T)$. The levels $f\inv(\{c'\})$ for $c'$ close to $c$ are tori $\Psi(C\times\T)$, where $C$ is a circle in $D$ homotopic to $0$. 
Denoting by $D^*$ the pointed disk, one shows (see \cite{BBM-10} and references therein) that $U\setm\CC=\Psi(D^*\times\T)$ is contained in a action-angle domain $\jA$.

$\bullet$ If $\CC$ has index $1$, there exists a neighborhood $U$ of $\CC$ such that $f\inv\{c\}\cap U$ is a stratified submanifold homeomorphic to  a ``fiber bundle'' with basis a circle and with fiber a ``cross''. 
The  connected component $\PP$ of $f\inv(\{c\})$ containing $\CC$ is a finite union of critical circles and cylinders $\T\times \R$ whose boundary is either made of one or two critical circles. All the critical circles contained in $\Pp$\index{$\Pp$} are homotopic and have index $1$. Such a stratified submanifold is called a \emph{polycycle}.
In \cite{F-88}, Fomenko assumes that a polycycle contains only one critical circle. In this case, we say that $\PP$ is a ``eight-level'', and we write $\8$-level. 

Finally in \cite{LM}, we have shown that if $\Tt$ is a critical torus for $f$, then $\Tt$ is contained in a action-angle domain $\ha\jA$. We denote by $\TT_c$ the set of all critical tori of $f$ and we  introduce the  domain $\ha{\RR}:=\RR\cup\TT_c$.
A connected component of $\ha{\RR}$ is  the connected component of an intersection $\jA:=\ha{\jA}\cap\EE$, where $\ha{\jA}$ is an action-angle domain.
 
Such a domain $\jA$ is diffeomorphic to $\T^2\times I$, where $I$ is an interval of $\R$ and it satisfies the following properties:
\begin{itemize}
\item there exist $a,b\in{\rm{Crit}}(f)$ with $a<b$ and $\jA= f\inv(]a,b[)$,
\item for all $x\in\,]a,b[$, $f\inv(x)\cap \Aa$ is diffeomorphic to $\T^2$,
\item there is a critical point of $f$ in each connected component of $\partial \jA$.
\end{itemize}
We say that $\jA$ is a \emph{maximal action-angle domain of} $(\EE,\phi_H,f)$.
This discussion can be summarized in the following lemma.

\begin{cor}
The energy level $\EE$ is a  finite disjoint  union  of  maximal action-angle domains  $\jA$, critical circles with index $0$ or $2$, $\8$-levels and Klein bottles. 
\end{cor}

Now, consider the Hamiltonian flow $\psi$ on an  energy level $\EE$ of a Hamiltonian function on a  $4$-dimensional symplectic manifold and let $\ga$ be  a periodic orbit of $\psi$ with period $T$. The eigenvalues of   $D\psi_T(q)$ do not depend on $q$ in $\ga$ and $X$ is an eigenvector for $D\psi_T$ associated with the eigenvalue $1$. We denote by $\lam_1, \lam_2$ the other two eigenvalues. Due to the conservation of the volume, one has $\lam_1\lam_2=1$. The closed orbit $\ga$ is said to be \emph{nondegenerate} if  $\lam_1$ and  $\lam_2$ are not equal to $1$. 
There are two types of nondegenerate closed orbit:
\begin{itemize}
\item \emph{elliptic} if $\lam_1$ and $\lam_2$ are complex conjugate numbers lying on the unit circle $\U$.
\item \emph{hyperbolic} if $(\lam_1,\lam_2)\in \R^2$ with $|\lam_i|\neq 1$.
\end{itemize}

Let us come back now to the Bott system $(\EE, X^H, f)$. We first observe that a periodic orbit that is nondegenerate must be a critical circle for $f$. 
We also see that an elliptic orbit is a critical circle with index $0$ or $2$ and that a hyperbolic orbit is a critical circle with index $1$.
Indeed, if $\ga$ is a hyperbolic periodic orbit, it possesses invariant manifolds $W^s$ and $W^u$ that meet transversaly along $\ga$: this is possible if and only if $\ga$ is contained in a $\8$-level. 

Conversely, a critical circle is not always a nondegenerate periodic orbit. This led us to the introduce the following definition.

\begin{Def}
The system $(\EE,\phi_H,f)$ is said to be \emph{dynamically coherent}  if the critical circles $\CC$ are nondegenerate periodic orbits.
\end{Def}

\begin{exm}
In \cite{L-1}, we show that the geodesic flows of generic tori of revolution are dynamically coherent systems with hyperbolic periodic orbits.
\end{exm}

We conclude this section with the following result proved in \cite{LM}:

\begin{thm}
Let $(\EE,\phi_H,f)$ be a dynamically coherent system. Then 
\[
\hp(\phi_H)\in\{0,1,2\}.
\]
Moreover, $\hp(\phi_H)=2$ if and only if $\phi_H$ possesses a hyperbolic orbit.
\end{thm}

\section{Proof of theorem A}

Let $(M,g)$ be a compact Riemannian manifold. The \emph{ geodesic action}  is  defined on the set of absolutely continuous curves $c:[a,b]\ma M$ by
\[
\Aa_g(c)=\demi\int_a^b||\dot{c}(t)||^2dt.
\]
A \emph{variation} of a curve $\ga:[a,b]\rit M$ is a differentiable map 
\[
\Ga:[a,b]\times]-\eps,\eps[\,\rit M,\quad \eps>0 
\]
such that $\Ga(t,0)=\ga(t)$ for all $t\in[a,b]$. The variation is  \emph{proper} if the endpoints are fixed, that is, $\Ga(a,s)=\ga(a)$ and $\Ga(b,s)=\ga(b)$ for all $s\in]-\eps,\eps[$. 
The vector field $V(t)=\Dp{\Ga}{s}(t,0)$ along $\ga$ is the \emph{variation field} of $\Ga$. It is called proper if $V(a)=V(b)=0$. 

The geodesic segments, that is, the projections $\ga:[a,b]\rit M$ of curves $[a,b]\rit T^*M : t\ma \phi_g^t(x)$, are the critical points of $\Aa_g$ in the following sense:   for any proper variation $\Ga$ of $\ga$, denoting by $\ga_s$ the curves $\ga_s:=\Ga(.\,,s)$, one has:
\[
\frac{d}{ds}\Aa_g(\ga_s)=0.
\]

A variation $\Ga:[a,b]\times]-\eps,\eps[\,\rit M$ of $\ga$ is  a \emph{geodesic variation} if the curves $\ga_s$ are geodesic segments for all $s$. 
One says that $\ga(b)$ is \textit{conjugate} to $\ga(a)$ (along $\ga$) if $\ga$ admits a geodesic variation whith  proper  variation field. 
The following proposition is a classical consequence of the possiblity of ``rounding the corners''.

\begin{prop}
Let $\ga:[a,b]\rit M$ be a segment of geodesic.
If there exists $\tau\in\,]a,b[$ such that $\ga(\tau)$ is conjugate to $\ga(a)$ along $\ga$, then, $\ga$ cannot be minimizing between $\ga(a)$ and $\ga(b)$.
\end{prop}

The following result was proved in the framework of the  theory of Hamilton-Jacobi equations (see \cite{F}). 
We denote by $X^g$ the geodesic vector field associated with the geodesic flow $\phi_g$ on $T^*M$.

\begin{thm}\label{graphmin}
Let $\ga$ be the projection of a solution of $X^g$ which is contained in a $C^1$ Lagrangian graph over $M$. Then, for any $a<b$ in $\R$, the curve $\ga_{|[a,b]}:[a,b]\rit M$ is a minimizer of $\Aa_g$.
In particular, $\ga$ does not have conjugate points.
\end{thm}

The following theorem was conjectured by Hopf who proved it in dimension 2. The proof for arbitrary dimensions is due to Burago and Ivanov (\cite{BI-94}).

\begin{thm}\textbf{\emph{(Hopf, Burago-Ivanov)}}\label{BI}. Assume that $g$ is a Riemannian metric on the torus $\T^n$ which does not have conjugate points. Then $g$ is  flat.
\end{thm}

A simple but remarkable consequence of this theorem is the following. Let $g$ be a Riemannian metric on $\T^n$ and $\EE_g$ be the unit cotangent bundle of $\T^n$.

 \begin{cor}\label{eq}
 The metric $g$ is flat if and only if $\EE_g$ is foliated by $\phi_g$-invariant tori that are $C^1$ graphs over $\T^n$. 
\end{cor}

Indeed, $(\Longrightarrow)$ is obvious. Conversely, assume that $\EE_g$ is foliated by $\phi_g$-invariant tori that are graphs over the base $\T^n$. Therefore, by Theorem \ref{graphmin}, a geodesic  cannot  have conjugate points, and by the Hopf theorem, $g$ is flat.
\vspace{0.2cm}

The proof of Theorem A is based on corollary \ref{eq}  and on the following  particular properties of pertubations of action-angle Hamiltonian systems with two degrees of freedom defined by a quadratic form. 

\begin{lem}\label{Graph}
Consider a positive definite quadratic form $h$. Let $H$ be the Hamiltonian function on $\T^2\times \R^2$ defined by $H(\th,r)=h(r)$.
Let $f:\T^2\times\R^2$ be a $C^5$ function with $||f||_{C^5}=1$. For $\eps>0$, we set $H_\eps: H+\eps f$, and we denote by $\phi_\eps$ the Hamiltonian flow associated with $H_\eps$. There exists $\eps_0$ such that for all $\eps\leq \eps_0$, 
\begin{enumerate}
\item there exist $\phi_\eps$-invariant tori in $H_\eps\inv(\{1\})$ that are the graphs of $C^1$ functions $\T^2\rit \R^2$,
\item if $\Tt\subset H_\eps\inv(\{1\})$ is  a $\phi_\eps$-invariant torus that is homotopic to $\T^2\times\{0\}$, then $\Tt$ is the graph of a Lipschitz function  $\T^2\rit \R^2$,
\item there does not exist any $\phi_\eps$-invariant Klein bottle in $H_\eps\inv(\{1\})$.
\end{enumerate}
\end{lem}

The proof of this lemma is given in section 4. Actually, (1) is exactly the result of the KAM theorem, and (3) is an easy consequence of the particular property for KAM tori to ``block'' the dynamics in $3$-dimensional energy levels and of the  form of the Hamiltonian $H$. The main interest of lemma \ref{Graph} is concentrated in (2).

\begin{proof}[Proof of Theorem A]
We denote by $H_{g_0}$ the geodesic Hamiltonian function on $T^*\T^2$ defined by $g_0$. 
For $\eps>0$, we denote  by $\UU_\eps$ the set of $C^5$ Riemannian metrics $g$ on  $\T^2$, such that $||g-g_0||_{C^5}\leq \eps$ (where $||\cdot||_{C^5}$ is the $C^5$-norm on the space of metrics on $\T^2$). For $g\in \UU_\eps$, we denote by $H_g$ the geodesic Hamiltonian function on $T^*\T^2$ defined by $g$. 
Fix a compact neighborhood $K$ of $H_{g_0}\inv(\{1\})$. There exists $c>0$, independent of $\eps$, such that $||H_{g_0}-H_g||_{K,C^5}\leq c\eps$ (where here, $||\cdot||_{K,C^5}$ is the $C^5$-norm on the space of functions $H :K\rit\R$).

By lemma \ref{Graph} (1), if $\eps$ is small enough, there exist invariant tori in $H_g\inv(\{1\})$ that are the graphs of  $C^1$ functions: $\T^2\rit \R^2$.

Assume now that $g\in \DD\CC$ and that $g$ is not flat. We denote by $f$ a nondegenerate Bott integral for $\phi_g$ in restriction to the unit cotangent bundle $\EE_g$ of $\T^2$. We want to see that $\EE_g$ contains a hyperbolic orbit. Since $H_g$ is dynamically coherent, it suffices to show that $\EE_g$ contains a $\8$-level.

By corollary \ref{eq}, at least one leaf  of the foliation induced by $f$ in $\EE_g$ is not a $C^1$ graph over $\T^2$. Let $\LL$ be such a leaf.
By lemma \ref{Graph} (3), $\LL$ is either an elliptic orbit, or a $\8$-level, or a torus,.
Note that  such tori are $C^1$ submanifolds.

If $\LL$ is a $\8$-level, the proof is complete. 
If $\LL$ is an elliptic orbit,  there exists a neighborhood $U$ of $\LL$, saturated for $f$, such that $\jA=U\setm \LL$ is a maximal action-angle domain. The domain $\jA$  is foliated by tori homotopic to $\LL$, these tori are obviously non homotopic to $\T^2\times\{0\}$.
Now if $\LL$ is a torus, by lemma \ref{Graph} (2), this torus is not homotopic to $\T^2\times\{0\}$. So $\LL$ is contained in a maximal action-angle domain  $\jA$  that is foliated by tori non homotopic to $\T^2\times\{0\}$.

Now, each  torus $\Tt$ that is a $C^1$ graph over $\T^2$ is contained in an action-angle domain $\jA'$ in $\EE_g$. Such a  domain $\jA'$ is foliated by tori homotopic to $\Tt$ (indeed, by lemma \ref{Graph} (2), these tori are $C^1$ graph  over $\T^2$). Therefore the boundary of one of the domains $\jA'$ must intersect the boundary of one of the previous domains $\jA$ and this intersection must be contained in a $\8$-level. 
\end{proof}

\section{Proof of lemma \ref{Graph}}

The proof  is based on two results of the theory of dynamical systems. The first one is the KAM theorem, describing the behavior of small perturbations of Hamiltonian systems in action-angle form, and the second one is the Birkhoff theorem, that deals with particular dynamical systems on the cylinder $\T\times I$ (where $I$ is an interval of $\R$), namely the twist maps.

In section 4.1. and 4.2 we briefly recall these two results. The proof of lemma \ref{Graph} is given in section 4.3.

\subsection{Basic KAM Theory }

In this section, $B$ is a bounded domain of $\R^n$ and $h:B\rit \R$ is a smooth function. 
We consider the Hamiltonian function  $H:\T^n\times B\rit \R:\,(\th,r)\ma h(r)$ and we denote by $\phi_H$ the Hamiltonian system associated with $H$. It is a system in action-angle form: the whole phase space is completely foliated by the invariant tori $\T^n\times\{r\}$. On each of these tori $\T^n\times\{r\}$ the Hamiltonian system $\phi_H$ induces a \emph{quasi-periodic motion}, that is, a linear flow with frequency $\om(r):=\nabla h(r)$. 
With any $\om\in \R^n\setm\{0\}$, we associate the following submodule of $\Z^n$:
\[
\MM(\om):=\{k\in\Z^n\,|\,\langle k,\om\rangle =0\}=(\R\om)^\perp\cap\Z^n,
\]
where $\langle\, , \rangle$ is the canonical scalar product on $\R^n$. 
A vector $\om\in\R^n\setm\{0\}$ is said to be resonant if $\MM(\om)\neq \emptyset$ and nonresonant otherwise.
If $\om$ is nonresonant,  all the orbits are dense. 

In an informal way, the KAM theory is the study of \emph{persistence} of certain nonresonant tori under small perturbations of the system, that is, for systems of the form $H+\eps f$, where $f$ is a ``sufficiently regular'' bounded function. 
The work of Kolmogorov, Arnol'd and Moser show that tori corresponding to ``strongly nonresonant'' frequency vectors persist if the pertubation is small enough and if $H$ satisfies some nondegeneracy conditions. This statement is made more precise below. There are a lot of articles, manuals, surveys on the fundamentals of KAM theory. We  refer for example  to \cite{ArnMM} and \cite{AKN}.

We say that a hypersurface  $\Ss\subset\R^n$  satisfies the transversality property (T) if for every point $u\in\Ss$, $\R u$ is transverse to $\Ss$.
For a  value $e$ of $h$, we set $\Om_e:=\om(H\inv(\{e\}))$.

\begin{Def}\label{iso1}
The Hamiltonian system associated with $H$ is said to be \emph{isoenergetically nondegenerate} in the neighborhood of $H\inv(\{e\})$ if there exists a neighborhood $V$ of $e$ in $h(B)$, such that $\om$ never vanishes in $H\inv(V)$ and that for  any $e\in V$,  $\Om_e$ satisfies (T).
\end{Def}
In particular, $\Om_e$ is $(n-1)$-dimensional. One easily checks that if $h$ is strictly convex, $H$ is isoenergetically nondegenerate in the neighborhood of its regular energy levels.

Let us recall  the following definition that explains what  means for a frequency vector to be ``strongly nonresonant''.

\begin{Def} Fix two positive numbers $\tau,\ga$. We say that $\om\in \R^n$ belongs to $\DD(\tau,\ga)$ if
\[
|\langle\om,k\rangle|\geq \frac{\ga}{||k||^{\tau}},\quad \forall k\in\, \Z^n\setm\{0\}.
\]
The set $\DD(\tau,\ga)$ is the set of \emph{Diophantine vectors of type $(\tau,\ga)$}. The union
\[
\DD(\tau):=\bigcup_{\ga>0}\DD(\tau,\ga)
\]
is the set of \emph{Diophantine vectors of type} $\tau$.
\end{Def}
It is well known that $\DD(\tau)$ has full measure when $\tau>n-1$.

\begin{thm}
\textbf{The KAM Theorem.} Let $k>2n$ and $f:\T^n\times \R^n$ be a $C^k$ function with $||f||_{C^k}=1$.  For $\eps>0$, we set $H_\eps:=H+\eps f$. We denote by $\phi_\eps$ the Hamiltonian flow associated to $H_\eps$.

Fix $\tau\in\, ]n-1,\demi k-1[$ and $\ga>0$. Let $e\in  H(B)$. Assume that the Hamiltonian system is isoenergetically nondegenerate in the neighborhood of $H\inv(\{e\})$. There exists $\eps_0>0$, such that, for all $0<\eps<\eps_0$ and for all $r\in h\inv(\{e\})$ such that $\om(r)\in \DD(\tau,\ga)\cap \Om_e$, there exists a $\phi_\eps$-invariant  torus $\Tt_{r}$ that satisfies:
\begin{itemize}
\item $\Tt_{r}$ is homotopic to $\T^n\times\{r\}$,
\item there exists $\de>0$ independent of $\eps$, $\Tt_{r}\subset \T^n\times [r-\de\sqrt{\eps},r+\de\sqrt{\eps}]$,
\item $H_\eps(\Tt_r)=e$.
\end{itemize}
If moreover $H$ is convex, $\Tt_r$ is the graph of a $C^1$ function $\ell:\T^2\rit \R^2$, with $||\ell-r||_{C_1}\leq c\sqrt{\eps}$, for a positive number $c$ independent of $\eps$.
\end{thm}

\subsection{Twist maps}

In this short section, we consider a particular class of maps on the cylinder $C:=\T\times [a,b[$ whith $a\in\R$ and $b\in\, ]a,+\infty]$. We could also consider cylinders of the form $C:=\{(\th,r(\th))\,|\, a\leq r(\th)\leq g(\th)\}$, where $g:\T\rit ]a,+\infty[$ is a continuous function. The cylinder $C$ is endowed with the canonical symplectic form $\Om:=d\th\wedge dr$.

 The universal covering  of $C$ is the strip $\R\times [a,b[$. We denote by $\pi$ the canonical projection $\pi:\R\times I\rit  C$. A lift of a map $f:C\rit C$ is a map $F:\R\times I\rit \R\times I$ such that $\pi\circ F=f\circ \pi$. 

\begin{Def}
An  \emph{area-preserving twist map} on $C$ is a diffeomorphism $f:C\rit C$ such that:
\begin{enumerate}
\item $f$ preserves the symplectic form,
\item $f$ preserves the boundary $\T\times\{a\}$ in the sense that there exists $\eps>0$ such that there exists $c\in ]a,b[$ such that if $(\th,r)\in\T\times[0,\eps[$, then $f(\th,r)\in\T\times [a,c[$,
\item \textbf{Torsion condition}: if $F$ is any lift of $f$ to $\R\times I$, then, $\Dp{F}{r}(x,r)>0$.
\end{enumerate}
\end{Def}

\begin{rem}\label{twisttorsion}
Let $f_\eps$ be a diffeomorphism $C\rit C$ such that $f_\eps^*\Om=\Om$ and that $||f-f_\eps||_{C^1}\leq \eps$, with $\eps>0$. Then, for $\eps$ small enough, $f_\eps$ is a twist map.
\end{rem}

Recall that if $f:X\rit X$ is a continuous map of a metric space $X$, a point $x\in X$ is said to be \emph{nonwandering for} $f$ if for any neighborhood $U$ of $x$, there exists an integer $n\in\N^*$ such that $f^n(U)\cap U\neq \emptyset$. We denote by $NW(f)$ the set of  nonwandering points. 

We refer to \cite{HK-95} for a proof of the following theorem due to Birkhoff.

\begin{thm}\textit{\textbf{Birkhoff's Theorem.}}
Let $f:C\rit C$ be an area-preserving twist map. Let $\Dd$ be an  $f$-invariant open relatively compact domain containing $\T\times\{a\}$ and with connected boundary $\partial\Dd$. Assume that $\Dd\subset NW(f)$. Then $\partial \Dd$ is  the graph of Lipschitz function $\T\rit\, ]a,b[$.
\end{thm}

\subsection{Proof of lemma \ref{Graph}.}

\begin{proof}[Proof of lemma \ref{Graph}] 
One juste has to prove (2) and (3).
We denote respectively by $X$ and $X_\eps$ the vector fields associated with the Hamiltonian functions $H$ and $H_\eps$, and by $\phi$ and $\phi_\eps$ their respective flows.
 
We begin with studying the geometry of $H_\eps\inv(\{1\})$.  Set $\Hh:=\{(r_1,r_2)\in \R^2\,|\, h(r_1,r_2)=1\}$. Therefore $H\inv(\{1\}):=\T^2\times \Hh$.
The vector field $X$ reads:
\[
\begin{aligned}
\dot{\th_1} & =2ar_1+cr_2,\qquad  &\dot{r_1} & = 0\\
\dot{\th_2} & = cr_1+2br_2,\qquad  &\dot{r_2} & =0.
\end{aligned}
\]
We denote by $D_1$ and $D_2$ the lines  with respective equations $ar_1+cr_2=0$, $br_2+cr_1=0$ in $\R^2$. Then, for $i=1,2$, $D_i$ intersects $\Hh$ at two points $A_i,B_i$. The connected components of $\Hh\setm\{A_1,B_1,A_2,B_2\}$ are 
\begin{align*}
\DD^{++}:=\{r\in\Hh\,|\,ar_1+cr_2>0,\, cr_1+br_2>0 \},\\ \DD^{+-}:=\{r\in\Hh\,|\,ar_1+cr_2>0,\, cr_1+br_2<0 \},\\
\DD^{--}:=\{r\in\Hh\,|\,ar_1+cr_2<0,\, cr_1+br_2<0 \},\\ \DD^{-+}:=\{r\in\Hh\,|\,ar_1+cr_2<0,\, cr_1+br_2>0 \}.
\end{align*}
In each of the domains $\DD\sss$, we fix a point $A_\eps\sss$, where $*$ stand for $+$ or $-$, such that $A_\eps\sss\in h\inv(1-\eps)$.
\vspace{0.2cm}

Now fix $\th^0\in \T^2$. 
Consider the surface $\Sig_{\th^0}$ defined by $\th=\th^0$. Since $\Sig_{\th^0}$ is transverse to $H\inv(\{1\})$, for $\eps$ small enough,
$\Sig_{\th^0}$ is transverse to $H_\eps\inv(\{1\})$ and
their intersection is a compact submanifold of dimension $1$, that is, a circle. Moreover, the projection $p(H_\eps\inv(\{1\}))$ on $\R^2$ of this circle is contained in the annulus delimited by the ellipses with equations $h(r)=1+\eps$ and $h(r)=1-\eps$.

Consider the four lines $D_1^+:=(A\pp_\eps,A\pl_\eps)$, $D_1^-:=(A\lp_\eps,A\mm_\eps)$, $D_2^+:=(A\lp_\eps,A\pp_\eps)$ and $D_2^-=(A\pl_\eps,A\mm_\eps)$.
We denote by $\DD_1^+$ the domain bounded by $D_1^+$ and the ellipses $h(r)=1+\eps$ and $h(r)=1-\eps$, which is contained in the set $\{ar_1+cr_2\geq 0\}$. We define in the same way the domains $\DD_1^-$, $\DD_2^+$ and $\DD_2^-$ (see the simplified drawing in Figure 1).
There exists $\al>0$ such that 
\begin{itemize}
\item $\DD_1^+\subset \{2ar_1+cr_2>\al\}$ and $\DD_1^-\subset \{2ar_1+cr_2<-\al\}$
\item $\DD_2^+\subset \{cr_1+2br_2>\al\}$ and $\DD_1^-\subset \{cr_1+2br_2<-\al\}$.
\end{itemize}

\begin{figure}[h]
\centering
\psset{xunit=0.75,yunit=0.75,runit=0.75,linewidth=0.01}
\begin{pspicture}(5cm,4cm)

\rput(-2,0){
\psellipse[linewidth=.015,fillstyle=solid,fillcolor=black!35](0,0)(3.4,2.4)
\psclip{\psframe[linestyle=none](-1.35,-3)(1.35,3)}
\psellipse[linewidth=.015,fillstyle=solid,fillcolor=white](0,0)(3.4,2.4)
\endpsclip
\psellipse[linewidth=.08,linestyle=dashed](0.15,-0.15)(3,2)
\psellipse[linewidth=.05](0,0)(3,2)
\psellipse[linewidth=.015,fillstyle=solid,fillcolor=white](0,0)(2.6,1.6)
\psline[linewidth=.015](1.35,-2.5)(1.35,2.5)
\rput(1.35,3){$D_1^+$}
\psline[linewidth=.015](-1.35,-2.5)(-1.35,2.5)
\rput(-1.35,3){$D_1^-$}
}
\psline{->}(-1,-2.2)(-1.9,-1.4)
\rput(-0.7, -2.2){$\Hh$}
\rput(-3.35,1.35){\pscircle[fillstyle=solid,fillcolor=black](0,0){.07}}
\rput(-0.65,1.35){\pscircle[fillstyle=solid,fillcolor=black](0,0){.07}}
\rput(-3.35,-1.35){\pscircle[fillstyle=solid,fillcolor=black](0,0){.07}}
\rput(-0.65,-1.35){\pscircle[fillstyle=solid,fillcolor=black](0,0){.07}}
\rput(-1.2,1){$A^{++}_\eps$}
\rput(-2.8,-1){$A^{--}_\eps$}
\rput(-1.2,-1){$A^{+-}_\eps$}
\rput(-2.8,1){$A^{-+}_\eps$}
\rput(-2,-3.5){The domains $\DD_1^+$ and $\DD_1^-$}

\rput(8,0){
\psellipse[linewidth=.015,fillstyle=solid,fillcolor=black!35](0,0)(3.4,2.4)
\psclip{\psframe[linestyle=none](-3.5,-1.35)(3.5,1.35)}
\psellipse[linewidth=.015,fillstyle=solid,fillcolor=white](0,0)(3.4,2.4)
\endpsclip
\psellipse[linewidth=.05](0,0)(3,2)
\psellipse[linewidth=.08,linestyle=dashed](0.15,-0.15)(3,2)
\psellipse[linewidth=.015,fillstyle=solid,fillcolor=white](0,0)(2.6,1.6)
\psline[linewidth=.015](-3.5,1.35)(3.5,1.35)
\rput(4,1.35){$D_2^+$}
\psline[linewidth=.015](-3.5,-1.35)(3.5,-1.35)
\rput(4,-1.35){$D_2^-$}
\psline{->}(-4.1,2)(-2.6,0.6)
\rput(-5.5,2.3){$p(H_\eps\inv(\{-1\}))\cap\Sig_{\th_0}$}
}
\rput(9.35,1.35){\pscircle[fillstyle=solid,fillcolor=black](0,0){.07}}
\rput(9.35,-1.35){\pscircle[fillstyle=solid,fillcolor=black](0,0){.07}}
\rput(6.65,-1.35){\pscircle[fillstyle=solid,fillcolor=black](0,0){.07}}
\rput(6.65,1.35){\pscircle[fillstyle=solid,fillcolor=black](0,0){.07}}
\rput(8.8,1){$A^{++}_\eps$}
\rput(8.8,-1){$A^{+-}_\eps$}
\rput(7.2,-1){$A^{--}_\eps$}
\rput(7.2,1){$A^{-+}_\eps$}
\rput(8,-3.5){The domains $\DD_2^+$ and $\DD_2^-$}
\end{pspicture}
\vskip3cm
\caption{The section $\Sig_{\th_0}$}
\end{figure}

The four domains $\T^2\times \DD_1^\pm$ and $\T^2\times \DD_2^\pm$ form a covering of $H_\eps\inv(\{1\})$. Moreover, since for all $(\th,r)\in\T^2\times \DD_1^+$, $\Dp{H_\eps}{r_1}(\th,r)\neq 0$, by the implicit function theorem there exist an interval $I_2^+$ and a function $R_1^+:\T^2\times I_2^+\rit \R$ such that:
\begin{equation}
H_\eps\inv(\{1\})\cap \left(\T^2\times \DD_1^+\right) =\{(\th,R_1^+(\th,r_2),r_2)\,|\,(\th,r_2)\in\T^2\times I_2^+\}.
\end{equation}
In the same way, there exist intervals $I_2^-,I_1^+$ and $I_1^-$ and functions $R_1^-,R_2^+$ and $R_2^-$ such that
\begin{align*} 
H_\eps\inv(\{1\})\cap\left(\T^2\times \DD_1^-\right) &=\{(\th,R_1^-(\th,r_2),r_2)\,|\,(\th,r_2)\in\T^2\times I_2^-\},\\ H_\eps\inv(\{1\})\cap\left(\T^2\times \DD_2^+\right) & = \{(\th,r_1,R_2^+(\th,r_1)\,|\,(\th,r_1)\in\T^2\times I_1^+\},\\ H_\eps\inv(\{1\})\cap\left(\T^2\times \DD_2^-\right)& =\{(\th,r_1,R_2^-(\th,r_1)\,|\,(\th,r_2)\in\T^2\times I_1^-\}.
\end{align*}

Since the set of Diophantine numbers $\DD(2)$ is dense in $\R^2$ and since $\om:r\ma\om(r)$ is a diffeomorphism, for any $\de>0$, there exists $r\pp\in B(A\pp,\de)$ such that $\om(r)\in \DD(2)$. Moreover, since $A\pp\in \DD_1^+\cap\DD_2^+$, we can assume that $\de$ is small enough so that $r\pp\in \DD_1^+\cap\DD_2^+$. Finally, since $\DD(2)$ is stable under multiplication by a real number and since $\om$ is linear, we can assume that $r\pp\in\Hh$.

Similarly, there exist $r\pl,r\lp$ and $r\mm$ in $\Hh\cap \DD_1^+\cap\DD_2^-$, $\Hh\cap \DD_1^-\cap\DD_2^+$  and $\Hh\cap \DD_1^-\cap\DD_2^-$ whose images by $\om$ are in $\DD(2)$. 
Let $\ga>0$ such that $\{\om(r\pp),\om(r\lp),\om(r\pl),\om(r\mm)\}\subset \DD(2,\ga)$.
\vspace{0.2cm}

By the KAM theorem, there exists $\eps_0>0$ such that for all $\eps<\eps_0$, there exist  tori $\Tt\pp,\Tt\pl,\Tt\lp$ and $\Tt\mm$ in $H_\eps\inv(\{1\})$, invariant  under the flow $\phi_\eps$, that are the graphs of  $C^1$-functions $g\sss:\T^2\rit\R^2$ with $||g\sss-r\sss||_{C^1}\leq c\sqrt{\eps}$.

We set $\Tt\sss:=\{(\th,r\sss(\th))\,|\,\th\in\T^2\}$. 
We choose $\eps<\eps_0$ small enough so that 
\begin{align*}
\Tt\pp\subset\T^2\times \left(\DD_1^+\cap\DD_2^+\right), & \quad\Tt\pl\in\T^2\times\left(\DD_1^+\cap\DD_2^-\right),\\
\Tt\lp\subset\T^2\times \left(\DD_1^-\cap\DD_2^+\right), &\quad \Tt\mm\in\T^2\times\left(\DD_1^-\cap\DD_2^-\right).
\end{align*}

Fix $\th^0\in \T^2$. The intersection $H_\eps\inv(\{1\})\cap \Sig_{\th^0}$ is the union of the curves:
\begin{itemize}
\item $\Cc_1^+(\th^0)$ with endpoints $r\pl(\th_0)$ and $r\pp(\th^0)$,  contained in $\T^2\times\DD_1^+$
\item $\Cc_1^-(\th^0)$ with endpoints $r\mm(\th_0)$ and $r\lp(\th^0)$,  contained in $\T^2\times\DD_1^-$
\item $\Cc_2^+(\th^0)$ with endpoints $r\lp(\th_0)$ and $r\pp(\th^0)$,  contained in $\T^2\times\DD_2^+$
\item $\Cc_2^-(\th^0)$ with endpoints $r\pl(\th_0)$ and $r\mm(\th^0)$,  contained in $\T^2\times\DD_2^-$.
\end{itemize}

The vector field $X^\eps$ reads
\[
\begin{aligned}
\dot{\th_1} & =2ar_1+cr_2+\eps \Dp{f}{r_1}(\th,r),\qquad  \dot{r_1} & =-\eps \Dp{f}{\th_1}(\th,r)\\
\dot{\th_2} & = cr_1+2br_2 +\eps \Dp{f}{r_2}(\th,r),\qquad  \dot{r_2} & =-\eps \Dp{f}{\th_2}(\th,r).
\end{aligned}
\]

We assume that $\eps$ is so small  that: 
\begin{align}\label{torsionr_1}
 2ar_1+cr_2+\eps\Dp{f}{r_1}(\th,r)>\demi \al, \:\:\forall (\th,r)\in \Cc_1^+\\
 2ar_1+cr_2+\eps\Dp{f}{r_1}(\th,r)<-\demi \al, \:\:\forall (\th,r)\in \Cc_1^-\\
 cr_1+2br_2+\eps\Dp{f}{r_2}(\th, r)>\demi \al, \:\:\forall (\th,r)\in \Cc_2^+,\\
 cr_1+2br_2+\eps\Dp{f}{r_2}(\th, r)<-\demi \al, \:\:\forall (\th,r)\in \Cc_2^-.
\end{align}

\begin{figure}[h]
\centering
\psset{xunit=0.75,yunit=0.75,runit=0.75,linewidth=0.01}
\begin{pspicture}(5cm,5cm)

\rput(3,2){
\psellipse[linewidth=.015](0,0)(3.4,2.4)
\psellipse[linewidth=.03](0,0)(3,2)
\psellipse[linewidth=.04,linestyle=dashed](0.15,-0.15)(3,2)
\psellipse[linewidth=.015](0,0)(2.6,1.6)
\psline[linewidth=.015](1.35,-3)(1.35,3)
\psline[linewidth=.015](-1.35,-3)(-1.35,3)
\psline[linewidth=.015](4,1.35)(-4,1.35)
\psline[linewidth=.015](-4,-1.35)(4,-1.35)
\rput(1.52,1.55){\pscircle[fillstyle=solid,fillcolor=black](0,0){.08}}
\rput(1.63,-1.83){\pscircle[fillstyle=solid,fillcolor=black](0,0){.08}}
\rput(-1.7,-1.7){\pscircle[fillstyle=solid,fillcolor=black](0,0){.08}}
\rput(-1.52,1.52){\pscircle[fillstyle=solid,fillcolor=black](0,0){.08}}
\psline[linewidth=.015]{->}(3.2,2.2)(1.52,1.52)
\psline[linewidth=.015]{->}(-3.2,2.2)(-1.52,1.52)
\psline[linewidth=.015]{->}(-3.3,-2.2)(-1.7,-1.7)
\psline[linewidth=.015]{->}(3.3,-2.2)(1.63,-1.83)
\rput(3.4,2.5){$r\pp(\th_0)$}
\rput(-3.4,2.5){$r\lp(\th_0)$}
\rput(-3.4,-2.5){$r\mm(\th_0)$}
\rput(3.4,-2.5){$r\pl(\th_0)$}
\psline[linewidth=.015]{->}(0,3.2)(0,1.8)
\psline[linewidth=.015]{->}(0,-3.5)(0,-2.1)
\psline[linewidth=.015]{->}(4.7,-0.1)(3.1,-0.1)
\psline[linewidth=.015]{->}(-4.7,-0.1)(-2.8,-0.1)
\rput(0,3.5){$\CC_2^+$}
\rput(0,-3.8){$\CC_2^-$}
\rput(5.2,-0.1){$\CC_1^+$}
\rput(-5.3,-0.1){$\CC_1^-$}
}
\end{pspicture}
\vskip1.5cm
\caption{The curves $\CC_1^+,\CC_1^-,\CC_2^+$ and $\CC_1^-$}
\end{figure}

We set: 
\begin{align*}
\Hh_{1,\eps}^+:=\bigcup_{\th\in \T^2}\Cc_1^+(\th), & \qquad   \Hh_{1,\eps}^-:=\bigcup_{\th\in \T^2}\Cc_1^-(\th)\\
\Hh_{2,\eps}^+:=\bigcup_{\th\in \T^2}\Cc_2^+(\th), & \qquad   \Hh_{2,\eps}^-:=\bigcup_{\th\in \T^2}\Cc_2^-(\th).
\end{align*}
The four  sets above are $3$-dimensional manifolds with boundary. They cover $H_\eps\inv(\{1\})$. Since the boundary of each of them is the disjoint union  of two KAM tori $\Tt^{**}$, they are invariant under the flow $\phi_\eps$.
Therefore   a $\phi_\eps$-invariant surface contained in $H_\eps\inv(\{1\})$ is necessarily contained in one of these submanifolds. 
\vspace{0.1cm}

Let $\LL$ be a $\phi_\eps$-invariant surface  and let us see that $\LL$ must be a torus, which will prove 3.
Assume that $\LL\subset \Hh_{1,\eps}^+$.
By (\ref{torsionr_1}), for any $\th_1^0\in \T$, the $3$-dimensional submanifold $\ha{S}_{\th_1^0}:=\{\th_1=\th_1^0\}$ is transverse to $\Hh_{1,\eps}^+$. 
We denote by $S_{\th_1^0}^+$ the symplectic surface  $S_{\th_1^0,\eps}^+:=\ha{S}_{\th_1^0}\cap \Hh_{1,\eps}^+$. 
We can assume without loss of generality that $r_2\pl <r_2\pp$. 

\begin{nota} 1) In what follows, we will only work in $\Hh_{1,\eps}^+$. We will omit the subscript $+$ and will write $\Hh_{1,\eps}$. 

\noindent In the same way, we set $S_{\th_1^0,\eps}:=S_{\th_1^0,\eps}^+$ and $r_2^+:=r_2\pp,\: r_2^-:=r_2\pl$.

2) We denote by $\Hh_1$ the intersection $H\inv(\{1\})\cap (\T^2\times \{2ar_1+br_2>\frac{1}{4}\al\})$. There exist an interval $I_2$ and a function $R_1:I_2\rit \R$ such that:
\begin{equation}
\Hh_1 =\{(\th,R_1(r_2),r_2)\,|\,(\th,r_2)\in\T^2\times I_2\}.
\end{equation}
Obviously, $I_2\supset [r_2^-(\th),r_2^+(\th)]$ for all $\th\in\T^2$.
\end{nota}

\noindent One has:
\[
S_{\th_1^0,\eps}:=\{(\th_1^0,\th_2,R_1^+(\th_1^0,\th_2,r_2),r_2)\,|\, r_2\in [r_2^-(\th_1^0,\th_2),r_2^+(\th_1^0,\th_2)]\},
\]
that is, $S_{\th_1^0,\eps}$ is parametrized by $(\th_2,r_2)\in \T\times [r_2^-(\th_1^0,\th_2),r_2^+(\th_1^0,\th_2)]$.

Since for all $\th_1^0\in\T$, $\LL$ is transverse to $S_{\th_1^0,\eps}$ in $\Hh_{1,\eps}$,\:  $\LL\cap S_{\th_1^0,\eps}$ is a  
 closed $1$-dimensional submanifold  $\CC(\th_1^0)$ possibly non connected.
Assume that $\CC(\th_1^0)$ is a finite union of circles $\Ga_1,\dots, \Ga_m$.  
The Poincaré return map $\wp_\eps:S_{\th_1^0,\eps}\rit S_{\th_1^0,\eps}$ with respect to the flow $\phi_\eps$  is well defined by (\ref{torsionr_1}). 
Necessarily, for any $q\in\{1,\dots,m\}$, there exists $p\neq q$ in $\{1,\dots,m\}$ such that $\wp_{\eps}(\Ga_q)\subset\Ga_p$. Since conversely, $\wp_{\eps}\inv(\Ga_q)\subset\Ga_p$, one has $\wp_{\eps}(\Ga_p)=\Ga_q$. 
We set 
\[
\wp_{\eps}(x)= \phi_\eps^{\tau(x)}(x).
\]
Observe that the map 
\[
\begin{matrix}
[0,1]\times \Ga_q & \rit & \LL\\
 (t,z) &\ma &  \phi_\eps^{t\tau(z)}(z)
\end{matrix}
\]
is a homotopy.

Now, for any $z\in \Ga_1$, $\phi_\eps(m\tau(z),z)\in \Ga_1$. Hence, the map
\[
\begin{matrix}
[0,1]\times \Ga_1 & \rit & \LL\\
 (t,z) &\ma &  \phi_\eps^{tm\tau(z)}(z)
\end{matrix}
\]
is surjective and $\LL$ is diffeomorphic to the quotient 
\[
[0,1]\times \Ga_1/\{(0,\phi(m\tau(z),z)=(1,z)\}.
\]
 Since the diffeomorphism $z\ma \phi_\eps(m\tau(z),z)$ is homotopic to the identity, $\LL$ is a torus and 3 is proved.
\vspace{0.1cm}


In remains to prove 2. Fix a $\phi_\eps$-invariant torus $\Tt$ in  $\Hh_{1,\eps}$. Assume that $\Tt$ is homotopic to $\T^2\times\{0\}$. 
For $\th_1\in \T$, we set $\CC(\th_1)=\Tt\cap S_{\th_1,\eps}$. We have already seen that $\CC(\th_1)$ is a finite union of circles $\Ga_1,\dots,\Ga_m$.

Observe that for all $\th_1^0,\th_1^1$ in $\T$, there exists a Poincaré map $P_{\th_1^0,\th_1^1}$ between the surfaces $S_{\th_1^0,\eps}$ and $S_{\th_1^1,\eps}$.
As before, one  checks that $P_{\th_1^0,\th_1^1}(\CC(\th_0^1))= \CC(\th_1^1)$ and that $P_{\th_0^1,\th_1^1}$ leads to a homotopy between $\CC(\th_0^1)$ and $\CC(\th_1^1)$.
Thus, all the submanifolds $\CC(\th_1^0)$ are homotopic (and in particular homologous) in $\Tt$.
Let us denote by $[\Ga]$ the common homology class (in $\Tt$) of the circles $\Ga_k$. Set $\CC:=\T\times\{0\}\subset\T^2$.  Clearly, $[\Ga]$ and $[\CC]$ are independent in $H_1(\Tt,\Z)$.
Since $\Tt$ is homotopic to $\T^2\times\{0\}$, the circles $\Ga_k$ must be essential in the cylinder $(\th_2,r_2)\in \T^2\times [r_2^-(\th_1^0,\th_2),r_2^+(\th_1^0,\th_2)]$, otherwise $\Tt$ would be homotopic to the curve $\{(m\th_1,0)\,|\,\th_1\in\T\}\times\{0\}\subset \T^2\times\{0\}$.


For $1\leq k\leq m$, we denote by $\Ii_k$ the domain in $S_{\th_1^0,\eps}$ bounded $\Ga_k$ and the lower boundary $\{(\th_2,r^{-}(\th_1^0,\th_2)\,|\,\th_2\in \T\}$.  
Now $\wp_\eps$ is symplectic and in particular preserves the areas, so all the $\Ii_k$ have the same area, that is, all the $\Ga_k$ coincide and the intersection $\CC(\th_1^0)$ between $\Tt$ with $S_{\th_1^0,\eps}$ is a single essential circle.
%

Let $L_\eps : S_{\th_1^0,\eps}\rit \T\times I: (\th_2,r_2)\ma (\th_2, r_2-r_2^-(\th_1^0,\th_2)+r_2^-)$ and set
\[
\begin{array}{llll}
\ha{\wp}_\eps & : \T\times \II & \longrightarrow &\: \T\times\II\\
& (\th_2,r_2) &\longmapsto &\: L_\eps\inv\circ\wp_\eps\circ L_\eps(\th_2,r_2),
\end{array}
\]
where $\T\times \II$ is the cylinder contained in $\T\times I$ whose boundaries are $\T\times\{r_2^-\}$ and the graph of the function $\th_2\ma r_2^+(\th_1^0,\th_2)-r_2^-(\th_1^0,\th_2)+r_2^-$.

We will apply Birkhoff's Theorem to  $\ha{\wp}_\eps$ to see  that $\CC(\th_1^0)$ is the graph of a function  $\T\rit \II$. 
Obviously, $\ha{\wp}_\eps$ preserves the symplectic form $d\th_2\wedge dr_2$ and the boundary $\T\times\{r_2^-\}$. One just has to check the torsion condition. To do this, we will see that $\ha{\wp}_\eps$ is $\sqrt{\eps}$-close in $C^1$-topology to the twist map defined by the Poincaré return map $\wp$ (with respect to $\phi$) associated with the surface $S_{\th_1^0}=\ha{S}_{\th_1^0}\cap \Hh_1$.
We first observe that there exists $\de>0$ independent of $\eps$, such that $|L_\eps -\Id|_{C^1,\T^2\times\II}\leq \de\sqrt{\eps}$.
One has:
\[
S_{\th_1^0}:=\{(\th_1^0,\th_2,R_1(r_2),r_2)\,|\, (\th_2,r_2)\in \T\times I_2\}.
\]
The Poincaré map $\wp$ reads:
\[
\wp(\th_2,r_2)= \left(\th_2+\frac{cR(r_2)+2br_2}{2aR(r_2)+cr_2},r_2\right)=(\wp_1(\th_2,r_2),\wp_2(\th_2,r_2)).
\]
Hence:
\[
 \Dp{\wp_1}{r_2}=\frac{(4ab-c^2)(R(r_2)-R'(r_2)r_2)}{(2aR(r_2)+cr_2)^2}.
\]
Using the fact that $r_2\ma aR(r_2)^2+br_2^2+cR(r_2)r_2$ is a constant function, one gets:
\[
R'(r_2)=-\frac{2br_2+cR(r_2)}{2aR(r_2)+cr_2}=-\frac{2br_2+cr_1}{2ar_1+cr_2}.
\]
Thus:
\[
 \Dp{\wp_1}{r_2}=-\frac{4ab-c^2}{(2aR(r_2)+cr_2)^2}\frac{r_1(2br_2+cr_1)+r_2(2ar_1+cr_2)}{2ar_1+cr_2}.
\]
Since $(r_1,r_2)\in \DD_1^+$, $2ar_1+cr_2>0$. Now $4ab-c^2=4\det h>0$. Finally, $r_1(2br_2+cr_1)+r_2(2ar_1+cr_2)=\langle r, n(r)\rangle$ where $n(r)$ is the normal vector pointing outwards the ellipse $\Hh$. Since $\Hh$ is convex, this scalar product  has constant sign. We can assume without loss of generalities, that $\langle r, n(r)\rangle>0$.
Thus, $\wp$ is an area-preserving twist map.  Its return-time map $\tau: (\th_2,r_2)\ma 2aR_1(r_2)+cr_2$ only depends on $r_2$. 

For $(\th_2,r_2)\in S_{\th_1^0,\eps}$, the return-time map $\tau_\eps(\th_2,r_2)$ is defined by:
\[
\int_0^{\tau_\eps(\th_2,r_2)}\dot{\th}_1(s)ds=\int_0^{\tau(\th_2,r_2)}2ar_1(s)+cr_2(s)+\eps\Dp{f}{r_1}(\th(s),r(s))ds=1.
\]
One easily checks that $||\tau-\tau_{\eps}||_{C^1,\T\times\II}\leq c\eps$ for  a suitable constant $c>0$ independent of $\eps$.

Set $J:=[0,\max_{\T^2\times\II}(\tau,\tau_\eps)]$ and $K:=J\times(\T\times \II)$. We denote by $\Phi$ and $\Phi_\eps$ the maps defined on $K$ by $\Phi(t,(\th,r))=\phi^t(\th,r)$ and $\Phi_\eps(t,(\th,r))=\phi_\eps^t(\th,r)$.
By the Gronwall lemma, there exists $k>0$ such that:
\[
||\Phi-\Phi_\eps||_{C^1(K)}\leq k||X-X_\eps||_{C^1(K)}.
\]
Therefore there exists $\ga>0$ independent of $\eps$, such that:
\begin{multline}
||\wp-\ha{\wp}_\eps||_{C^1(\T\times\II)}\leq \ga\sup\bigg(||X-X_\eps||_{C^1(K)}+||\tau-\tau_\eps||_{C^1(J)}\\
+||\Id-L_\eps\inv||_{C^1(\T\times\II)}+||\Id-L_\eps||_{C^1(\T\times\II)}\bigg),
\end{multline}
that is, there exists $\ga'$ independent of $\eps$, such that:
\[
||\wp-\ha{\wp}_\eps||_{C^1}\leq \ga'\sqrt{\eps}.
\]
Then for $\eps$ small enough, $\ha{\wp}_\eps$ satisfies the torsion condition by remark \ref{twisttorsion}. Therefore we can apply Birkhoff's Theorem and $\CC(\th_1^0)$ is the graph of a Lipschitz fonction $R_{\th_1^0}:\T\rit \R^2$.
As a consequence, 
\begin{align*}
\LL:=\bigcup_{\th_1\in\T}\CC(\th_1) & =\bigcup_{\th_1\in\T}\{(\th_2,R_{\th_1}(\th_2))\,|\,\th_2\in\T\} \\
 &=\{(\th_1,\th_2,R(\th_1,\th_2))\,|\,(\th_1,\th_2)\in\T^2\}.
\end{align*}
and  $R:(\th_1,\th_2)\ma R(\th_1,\th_2)$ is continuous.
The same argument holds true in each of the domains $\Hh_{1,\eps}^-, \Hh_{2,\eps}^-$ and $\Hh_{2,\eps}^+$, which concludes the proof.
\end{proof}

\bibliographystyle{alpha}
\bibliography{biblioflat}

\def\cprime{$'$}
\begin{thebibliography}{BBM10}

\bibitem[AKN06]{AKN}
Vladimir~I. Arnold, Valery~V. Kozlov, and Anatoly~I. Neishtadt.
\newblock {\em Mathematical aspects of classical and celestial mechanics},
  volume~3 of {\em Encyclopaedia of Mathematical Sciences}.
\newblock Springer-Verlag, Berlin, third edition, 2006.
\newblock [Dynamical systems. III], Translated from the Russian original by E.
  Khukhro.

\bibitem[Arn99]{ArnMM}
V.~I. Arnol{\cprime}d.
\newblock {\em Mathematical methods of classical mechanics}, volume~60 of {\em
  Graduate Texts in Mathematics}.
\newblock Springer-Verlag, New York, 199?
\newblock Translated from the 1974 Russian original by K. Vogtmann and A.
  Weinstein, Corrected reprint of the second (1989) edition.

\bibitem[BBM10]{BBM-10}
A.~V. Bolsinov, A.~V. Borisov, and I.~S. Mamaev.
\newblock Topology and stability of integrable systems.
\newblock {\em Uspekhi Mat. Nauk}, 65(2(392)):71--132, 2010.

\bibitem[BCG95]{BCG-95}
G.~Besson, G.~Courtois, and S.~Gallot.
\newblock Entropies et rigidit\'es des espaces localement sym\'etriques de
  courbure strictement n\'egative.
\newblock {\em Geom. Funct. Anal.}, 5(5):731--799, 1995.

\bibitem[BCG96]{BCG-96}
G.~Besson, G.~Courtois, and S.~Gallot.
\newblock Minimal entropy and {M}ostow's rigidity theorems.
\newblock {\em Ergodic Theory Dynam. Systems}, 16(4):623--649, 1996.

\bibitem[BI94]{BI-94}
D.~Burago and S.~Ivanov.
\newblock Riemannian tori without conjugate points are flat.
\newblock {\em Geom. Funct. Anal.}, 4(3):259--269, 1994.

\bibitem[Fat]{F}
A.~Fathi.
\newblock {\em The weak {KAM} theorem in {L}agrangian dynamics}.

\bibitem[Fom88]{F-88}
A.~T. Fomenko.
\newblock {\em Integrability and nonintegrability in geometry and mechanics},
  volume~31 of {\em Mathematics and its Applications (Soviet Series)}.
\newblock Kluwer Academic Publishers Group, Dordrecht, 1988.
\newblock Translated from the Russian by M. V. Tsaplina.

\bibitem[Kat88]{Kat}
A.~Katok.
\newblock Four applications of conformal equivalence to geometry and dynamics.
\newblock {\em Ergodic Theory Dynam. Systems}, 8$^*$(Charles Conley Memorial
  Issue):139--152, 1988.

\bibitem[KH95]{HK-95}
A.~Katok and B.~Hasselblatt.
\newblock {\em Introduction to the modern theory of dynamical systems},
  volume~54 of {\em Encyclopedia of Mathematics and its Applications}.
\newblock Cambridge University Press, Cambridge, 1995.
\newblock With a supplementary chapter by Katok and Leonardo Mendoza.

\bibitem[L]{L-1}
C~Labrousse.
\newblock Polynomial growth of the volume of balls for zero entropy geodesic
  systems.

\bibitem[LM]{LM}
C.~Labrousse and J-P.~Marco.
\newblock Polynomial entropies for Bott nondegenerate Hamiltonian systems.

\bibitem[Man79]{M-79}
A~Manning.
\newblock Topological entropy for geodesic flows.
\newblock {\em Ann. of Math.}, 110:567--573, 1979.

\bibitem[Mar93]{Mar-93}
J.-P. Marco.
\newblock Obstructions topologiques à l'intégrabilité des flots géodésiques en
  classe de bott.
\newblock {\em Bull. Sc math.}, 117:185--209, 1993.

\bibitem[Mar09]{Mar-09}
J.-P. Marco.
\newblock Dynamical complexity and symplectic integrability.
\newblock ArXiv, 2009.

\bibitem[Pat91]{Pat-dim4}
G.~Paternain.
\newblock Entropy and completely integrable {H}amiltonian systems.
\newblock {\em Proc. Amer. Math. Soc.}, 113(3):871--873, 1991.

\end{thebibliography}

\end{document}